\documentclass{amsart}

\usepackage{amssymb}

\usepackage[matrix,arrow,curve]{xy}

\theoremstyle{plain}

\newtheorem{theorem}{Theorem}[section]

\newtheorem{corollary}[theorem]{Corollary}

\theoremstyle{definition} 

\newtheorem{definition}[theorem]{Definition}

\newtheorem{example}[theorem]{Example}

\theoremstyle{remark}

\newtheorem{problem}[theorem]{Problem}

\renewcommand{\phi}{\varphi}

\newcommand{\initial}\lessdot

\def\?{?\vadjust

{\vbox to 0pt{\vskip-7pt\hbox to 1.1\hsize{\hfill\huge ?!}}}}

\newcommand{\be}{\begin{enumerate}}
\newcommand{\ee}{\end{enumerate}}

\usepackage{enumerate}
\usepackage{verbatim}
\renewcommand{\epsilon}{\varepsilon}

 \def\nfork{\setbox0\hbox{$\bigcup$}%
 \setbox1=\hbox to \wd0{\hfil\vrule width 0.7pt depth 2pt height 7.5pt\hfil}%
 \wd1=0cm\relax\box1\box0}

\begin{document}

\title{The First Time KE is Broken Up}

\author{Adi Jarden}
\email[Adi Jarden]{jardena@ariel.ac.il}
\address{Department of Mathematics.\\ Ariel University \\ Ariel, Israel}

\maketitle

\tableofcontents

\today

\begin{abstract}
A relevant collection is a collection, $F$, of sets, such that each set in $F$ has the same cardinality, $\alpha(F)$. A Konig Egervary (KE) collection is a relevant collection $F$, that satisfies $|\bigcup F|+|\bigcap F|=2\alpha(F)$. An hke (hereditary KE) collection is a relevant collection such that all of his non-empty subsets are KE collections. The motivation for the study of hke collections comes from results of Jarden, Levit and Mandrescu \cite{jlm} and \cite{dam}. In \cite{hke}, Jarden characterize hke collections. 

Let $\Gamma$ be a relevant collection such that $\Gamma-\{S\}$ is an hke collection, for every $S \in \Gamma$. We study the difference between $|\bigcap \Gamma_1-\bigcup \Gamma_2|$ and $|\bigcap \Gamma_2-\bigcup \Gamma_1|$, where $\{\Gamma_1,\Gamma_2\}$ is a partition of $\Gamma$. We get new characterizations for an hke collection and for a KE graph.      
\end{abstract}

\section{Introduction}
We first present three definitions relating to collections and then discuss their connections to the study of graphs.

\begin{definition}
A \emph{relevant collection} is a finite collection of finite sets such that the number of elements in each set in $F$ is a constant positive integer, denoted $\alpha(F)$. When $F$ is clear from the context, we omit it, writting $\alpha$.
\end{definition}

While in \cite{dam}, a KE collection is defined with respect to a graph, the definition here and in \cite{hke} has no an explicit connection to a graph. The two different definitions are contradict.
\begin{definition}
Let $F$ be a relevant collection. $F$ is said to be a \emph{Konig Egervary collection} (KE collection in short), if the following equality
 holds: $$|\bigcup F|+|\bigcap F|=2\alpha.$$
\end{definition}

\begin{definition}\label{definition of a KE collection}
An \emph{hereditary Konig Egervary collection} (hke collection in short) is a collection of sets, $F$, such that for some positive integer, $\alpha$, the equality
$$|\bigcup \Gamma|+|\bigcap \Gamma|=2\alpha$$ holds for every non-empty sub-collection, $\Gamma$, of $F$. We call this $\alpha$, $\alpha(F)$. We may omit $F$, where it is clear from the context.
\end{definition}
 
For example, every relevant collection of cardinality $2$ at most is an hke collection. 
 
We now discuss graphs. A set $S$ of vertices in a graph is independent if no two vertices from $S$ are adjacent. A maximum independent set is an independent set of maximal cardinality. $\Omega(G)$ denotes the set of maximum independent sets in $G$. A matching is a set of pairwise non-incident edges. The graph $G$ is known to be a Konig-Egervary (KE in short) graph if $\alpha(G) + \mu(G)= |V(G)|$, where $\alpha(G)$ denotes the size of a maximum independent set and $\mu(G)$ is the cardinality of a maximum matching.

As a motivation for the study of hke collections, we restate two theorems:  
\begin{theorem}\cite[Theorem 2.6]{dam}\label{dam 2.6}
$G$ is a KE graph if and only if for some hke collection of maximum independent sets, there is a matching $$M:V(G)-\bigcup \Gamma \to \bigcap \Gamma.$$
\end{theorem}


\begin{theorem}\cite[Theorem 6.3]{hke}
A collection $F$ of sets is an hke collection if and only if $F$ is a KE collection and $F \subseteq \Omega(G)$ for some graph $G$.
\end{theorem}

\begin{definition}
For every collection $\Gamma$ of sets, we define $$e(\Gamma)=|\bigcup \Gamma|+|\bigcap \Gamma|.$$
\end{definition}

A relevant collection $\Gamma$ is a KE collection if and only if $e(\Gamma)=2\alpha$. 
Theorem \ref{monotonicity} shows the importance of the function $e$. In order to state Theorem \ref{monotonicity}, we present the following definition: 
\begin{definition}
Let $\Gamma,\Gamma^{\prime}$ be two collections of sets. We set $\Gamma
^{\prime}\vartriangleleft\Gamma$ if $%
{\displaystyle\bigcup}
\Gamma^{\prime}\subseteq%
{\displaystyle\bigcup}
\Gamma$ and $%
{\displaystyle\bigcap}
\Gamma\subseteq%
{\displaystyle\bigcap}
\Gamma^{\prime}$.
\end{definition}

Obviously, if $\Gamma' \subseteq \Gamma$ then $\Gamma' \vartriangleleft \Gamma$. 

\begin{theorem}\label{monotonicity}\cite[Theorem 2.6]{jlm}
Let $G$ be a graph. Let $\Gamma \subseteq \Omega(G)$ and let $\Gamma'$ be a collection of independent sets. If $\Gamma' \vartriangleleft \Gamma$ then $e(\Gamma') \leq e(\Gamma)$.
\end{theorem}


\section{The First Cardinality without the KE Property}

By Theorems \ref{broken} and \ref{may 2015 2}, the KE property is broken up in a unified form.

\begin{theorem}\label{broken}
Let $\Gamma$ be a collection of sets. Assume that there is a positive integer $\beta$ such that for every $S \in \Gamma$, we have $e(\Gamma-\{S\})=\beta$. Then there is an integer $m$ such that for every $S \in \Gamma$, the following equality holds:
$$|S-\bigcup (\Gamma-\{S\})|-|\bigcap (\Gamma-\{S\})-S|=m.$$

If, in addition, for some (equivalently, for each) $S \in \Gamma$, the collection $\Gamma-\{S\}$ is a KE collection then $m=0$ if and only if $\Gamma$ is a KE collection. 
\end{theorem}

\begin{proof}
For every $S \in \Gamma$, we have
$$e(\Gamma)-e(\Gamma-\{S\})=|S-\bigcup (\Gamma-\{S\})|-|\bigcap (\Gamma-\{S\})-S|.$$
By assumption, $e(\Gamma)-e(\Gamma-\{S\})$ does not depend on $S$. So there exists $m$ as needed. 

Assume that for some $S \in \Gamma$, the collection $\Gamma-\{S\}$ is a KE collection. Then $\beta=2\alpha$ and
$$e(\Gamma)-2\alpha=e(\Gamma)-e(\Gamma-\{S\})=m.$$

Therefore $m=0$ if and only if $e(\Gamma)=2\alpha$.
\end{proof}

\begin{example}
$$\Gamma=\{\{1\},\{2\},\{3\}\}.$$
$\alpha(\Gamma)=1$. $e(\Gamma)=3$ and for every $S \in \Gamma$, $e(\Gamma-\{S\})=2$. So $m=1$. 
\end{example}

By the following example, $m$ can be a negative number.
\begin{example}
$\Gamma=\{S_1,S_2,S_3\}$, where $S_1=\{1,2,3,4,5\}$, $S_2=\{4,5,6,7,8\}$ and $S_3=\{3,6,7,8,9\}$. For this non-KE collection, $\alpha=5$, $k=3$ and $m=-1$, where $k$ and $m$ are the parameters that appear in Proposition \ref{broken}.
\end{example}
 
\begin{corollary}
Let $\Gamma=\{A,B,C\}$ be a relevant collection of three different sets. Then $\Gamma$ is an hke collection if and only if 
$$|A-B-C|=|B \cap C-A|.$$  
\end{corollary}

\begin{proof} 
For every $S \in \Gamma$, the relevant collection $\Gamma-\{S\}$ is an hke collection, because its cardinality is two. So by Theorem \ref{broken}, there is an $m$ such that for every $S \in \Gamma$, the following equality holds:
$$|S-\bigcup (\Gamma-\{S\})|-|\bigcap (\Gamma-\{S\})-S|=m.$$

Moreover $m=0$ if and only if $\Gamma$ is an hke collection. But the equality that appears in the corollary holds if and only if $m=0$. 
\end{proof}

While in Theorem \ref{broken}, the main issue is
$$|S-\bigcup (\Gamma-\{S\})|-|\bigcap (\Gamma-\{S\})-S|,$$
in Theorem \ref{may 2015 2} (below), we concentrate on 
$$|\bigcap \Gamma_1-\bigcup \Gamma_2|-|\bigcap \Gamma_2-\bigcup \Gamma_1|,$$
where $\{\Gamma_1,\Gamma_2\}$ is a partition of $\Gamma$ into two non-empty subcollections.

Recall,
\begin{theorem}\cite[Theorem 2.13]{hke}\label{the duality in an hke collection}
Let $\Gamma$ be a collection of sets. Then $\Gamma$ is an hke collection if and only if the equality
$$|\bigcap \Gamma_1-\bigcup \Gamma_2|=|\bigcap \Gamma_2-\bigcup \Gamma_1|$$ holds for every partition $\{\Gamma_1,\Gamma_2\}$ of $\Gamma$  into two non-empty subcollections.
\end{theorem}

If we substitute $\Gamma_1=\{S\}$ and $\Gamma_2=\Gamma-\{S\}$ in Theorem \ref{may 2015 2}, then we get a result that is known by Theorem \ref{broken}. In a sense, Theorem \ref{may 2015 2} is more general than Theorem \ref{broken}, because Theorem \ref{may 2015 2} deals also with $\Gamma_1$ that includes more than one set. But in Theorem \ref{may 2015 2}, we have an extra assumption: $\Gamma-\{S\}$ is an hke collection.

\begin{theorem}\label{may 2015 2}
Let $\Gamma$ be a collection of sets such that $\Gamma-\{S\}$ is an hke collection, for each $S \in \Gamma$. Then there exists an integer $m$ such that the equality
$$|\bigcap \Gamma_1-\bigcup \Gamma_2|-|\bigcap \Gamma_2-\bigcup \Gamma_1|=(-1)^{|\Gamma_1|+1}m$$
holds for every partition $\{\Gamma_1,\Gamma_2\}$ of $\Gamma$ into two non-empty subcollections.  
\end{theorem}

\begin{proof}
The assumption of Theorem \ref{broken} holds for $\beta=2\alpha$. So we can find an integer $m$ satisfying the conclusion of Theorem \ref{broken}. We prove the needed equality by induction on $l=|\Gamma_1|$. For $l=1$, it holds by Theorem \ref{broken}. Assume that $1<l<|\Gamma|$ and $\{\Gamma_1,\Gamma_2\}$ is a partition of $\Gamma$ such that $|\Gamma_1|=l$. Fix $S \in \Gamma_1$. Let $\Gamma_1'=\Gamma_1-\{S\}$ and $\Gamma_2'=\Gamma_2 \cup \{S\}$. By set theoretical considerations, we get
 
$$(1) |\bigcap \Gamma_1-\bigcup \Gamma_2|=|\bigcap \Gamma_1'-\bigcup \Gamma_2|-|\bigcap \Gamma_1'-\bigcup \Gamma_2'|$$ and $$(2) |\bigcap \Gamma_2-\bigcup \Gamma_1|=|\bigcap \Gamma_2-\bigcup \Gamma_1'|-|\bigcap \Gamma_2'-\bigcup \Gamma_1'|.$$ 

By assumption, $\Gamma-\{S\}$ is an hke collection. $\{\Gamma_1',\Gamma_2\}$ is a partition of $\Gamma-\{S\}$ into two non-empty subcollections. So by Theorem \ref{the duality in an hke collection}, we have:
$$(3) |\bigcap \Gamma_1'-\bigcup \Gamma_2|=|\bigcap \Gamma_2-\bigcup \Gamma_1'|.$$

Since $|\Gamma_1'|=l-1$ and $\{\Gamma_1',\Gamma_2'\}$ is a partition of $\Gamma$ into two non-empty subcollections, by the induction hypothesis, the following equality holds:
$$(4) |\bigcap \Gamma_1'-\bigcup \Gamma_2'|-|\bigcap \Gamma_2'-\bigcup \Gamma_1'|=(-1)^{|\Gamma_1|}m$$

By the subtraction of Equality $(2)$ from Equality $(1)$, by Equality $(3)$ and by Equality $(4)$, we get:
$$|\bigcap \Gamma_1-\bigcup \Gamma_2|-|\bigcap \Gamma_2-\bigcup \Gamma_1|=(-1)^{|\Gamma_1|+1}m.$$
\end{proof}

\begin{corollary}
Let $\Gamma$ be a collection of sets such that for each $S \in \Gamma$, $\Gamma-\{S\}$ is an hke collection. If $|\Gamma|$ is an even then $\Gamma$ is an hke collection.
\end{corollary}

\begin{proof}
Let $m$ be the integer that exists by Theorem \ref{may 2015 2}. Since $|\Gamma|$ is even, we can find a partition $\{\Gamma_1,\Gamma_2\}$ of $\Gamma$ such that $|\Gamma_1|=|\Gamma_2|$. By Theorem \ref{may 2015 2}, the following equality holds:
$$|\bigcap \Gamma_1-\bigcup \Gamma_2|-|\bigcap \Gamma_2-\bigcup \Gamma_1|=(-1)^{|\Gamma_1|+1}m.$$
But if we replace between $\Gamma_1$ and $\Gamma_2$ then we get the following equality:
$$|\bigcap \Gamma_2-\bigcup \Gamma_1|-|\bigcap \Gamma_1-\bigcup \Gamma_2|=(-1)^{|\Gamma_2|+1}m=(-1)^{|\Gamma_1|+1}m.$$

By the last two equalities, we get $x-y=y-x$, or equivalently, $x-y=0$, where
$x=|\bigcap \Gamma_1-\bigcup \Gamma_2|$ and $y=|\bigcap \Gamma_2-\bigcup \Gamma_1|$.
  
Hence, $m=0$. So $\Gamma$ is an hke collection.
\end{proof}

In Theorem \ref{may 2015 2}, $m$ may be negative. But if we assume, in addition, that $\Gamma$ is included in $\Omega(G)$ for some graph $G$ then by the following corollary, $m$ cannot be negative.
\begin{corollary} 
Let $G$ be a graph Let $\Gamma$ be a subset of $\Omega(G)$ such that $\Gamma-\{S\}$ is an hke collection, for each $S \in \Gamma$. Then there exists an integer $0 \leq m$ such that the equality
$$|\bigcap \Gamma_1-\bigcup \Gamma_2|-|\bigcap \Gamma_2-\bigcup \Gamma_1|=(-1)^{|\Gamma_1|+1}m$$
holds for every partition $\{\Gamma_1,\Gamma_2\}$ of $\Gamma$ into two non-empty subcollections.  
\end{corollary}

\begin{proof}
By Theorem \ref{may 2015 2} there exists an integer $m$ satisfying the equality. We have to prove that $m$ cannot be negative. Fix $S \in \Gamma$. By Theorem \ref{monotonicity}, $$e(\Gamma-\{S\}) \leq e(\Gamma).$$
Hence, $$0 \leq e(\Gamma)-e(\Gamma-\{S\})=m.$$
\end{proof}

\section{New Characterizations} 
Theorem \ref{less partitions are enough} presents a new characterization of an hke collection. Corollary \ref{corollary for a graph} is a version of Theorem \ref{less partitions are enough} for graphs. 

While in Theorem \ref{the duality in an hke collection}, in order to prove that $F$ is an hke collection, we have to consider every partition into non-empty subcollections, here we consider much less partitions. 

\begin{theorem}\label{less partitions are enough}
Let $F$ be a collection of sets. Then $F$ is an hke collection if and only if for every subcollection $\Gamma$ of $F$, there exists a partition $\{\Gamma_1,\Gamma_2\}$ of $\Gamma$ into two non-empty subcollections, such that the following equality holds:
$$|\bigcap \Gamma_1-\bigcup \Gamma_2|=|\bigcap \Gamma_2-\bigcup \Gamma_1|.$$
\end{theorem}

\begin{proof}
If $F$ is an hke collection then the conclusion holds by Theorem \ref{the duality in an hke collection}. 

Conversely, assume that $F$ is not an hke collection. So there is a non-KE subcollection $\Gamma$ of $F$ of minimal cardinality. So $\Gamma-\{S\}$ is an hke collection, for every $S \in \Gamma$. By Theorem \ref{may 2015 2}, there is an integer $m$, such that $$|\bigcap \Gamma_1-\bigcup \Gamma_2|-|\bigcap \Gamma_2-\bigcup \Gamma_1|=(-1)^{|\Gamma_1|+1}m$$
holds for every partition $\{\Gamma_1,\Gamma_2\}$ of $\Gamma$ into two non-empty subcollections. 

$m \neq 0$, because otherwise by Theorem \ref{the duality in an hke collection}, $\Gamma$ is an hke collection, contradicting the choice of $\Gamma$. So there is \emph{no} partition $\{\Gamma_1,\Gamma_2\}$ of $\Gamma$ into two non-empty subcollections, such that the following equality holds:
$$|\bigcap \Gamma_1-\bigcup \Gamma_2|=|\bigcap \Gamma_2-\bigcup \Gamma_1|.$$ 
\end{proof}

By the following example, it is possible that the graph induced by three specific maximum independent sets will be a KE graph, while the graph induced by another three maximum independent sets will not be KE.
\begin{example}
Let
$$V(G)=\{1,2,3,4,5,6,7\} \text{ and}$$ $$E(G)=\{(14),(15),(16),(17),(24),(36),(46),(47),(57),(67)\}.$$ $\alpha(G)=3$ and $\mu(G)=3$. So $G$ is not a KE graph. Define $S_1=:\{1,2,3\}$, $S_2=:\{3,4,5\}$, $S_3=:\{2,5,6\}$ and $S_4=:\{2,3,7\}$. $S_1,S_2,S_3$ and $S_4$ are maximum independent sets. $G_1=:G[\{1,2,3,4,5,6\}]=G[S_1 \cup S_2 \cup S_3]$ is KE, but $G_2=:G[\{1,2,3,4,5,7\}]=G[S_1 \cup S_2 \cup S_4]$ is not KE: both $G_1$ and $G_2$ have six elements and $\alpha=3$. But while $\mu_{G_1}=3$, $\mu_{G_2}=2$ (the vertex $3$ is isolated in $G_2$).  
\end{example}

 \begin{corollary}\label{corollary for a graph}
$G$ is a KE graph if and only if for some $F \subseteq \Omega(G)$, there is a matching from $V(G)-\bigcup F$ into $\bigcap F$ and for every subcollection $\Gamma$ of $F$ there is a partition $\{\Gamma_1,\Gamma_2\}$ of $\Gamma$ into two non-empty sub-collctions   such that the following equality holds: $$|\bigcap \Gamma_1-\bigcup \Gamma_2|=|\bigcap \Gamma_2-\bigcup \Gamma_1|.$$
\end{corollary}

\begin{proof}
Assume that $G$ is a KE graph. By Theorem \ref{dam 2.6}, there is an hke subcollection $F$ of $\Omega(G)$ and there is a matching as needed. So by Theorem \ref{less partitions are enough}, for some partitions $\{\Gamma_1,\Gamma_2\}$ of $F$ into two non-empty subcollections, the following equality holds:
$$|\bigcap \Gamma_1-\bigcup \Gamma_2|=|\bigcap \Gamma_2-\bigcup \Gamma_1|.$$

Conversely, assume that some $F \subseteq \Omega(G)$ satisfies the condition in the corollary. By Theorem \ref{less partitions are enough}, $F$ is an hke collection. So by Theorem \ref{dam 2.6}, $G$ is a KE graph. 
 \end{proof}


\begin{problem}
Characterize the collections $\Gamma$ of sets such that $\Gamma=\Omega(G)$, for some graph $G$.
\end{problem}

Connections between the current paper, \cite{critical}, \cite{criticallema} and \cite{criticalshort} should be studied. 

\bibliographystyle{amsplain}
\bibliography{..//..//lit}

\end{document}